\documentclass[12pt,reqno]{amsart}
\usepackage{amssymb,amsmath,amsthm,enumerate,bbm}
\usepackage[a4paper]{geometry}

\DeclareMathOperator{\Tr}{Tr}

\DeclareMathOperator{\meas}{meas}

\newcommand{\abs}[1]{\lvert#1\rvert}
\newcommand{\Abs}[1]{\left\lvert#1\right\rvert}
\newcommand{\norm}[1]{\lVert#1\rVert}

\newcommand{\jap}[1]{\langle#1\rangle}

\newcommand{\bbR}{{\mathbb R}}
\newcommand{\bbC}{{\mathbb C}}

\newcommand{\bbN}{{\mathbb N}}
\newcommand{\bbZ}{{\mathbb Z}}

\newcommand{\calL}{\mathcal{L}}

\newcommand{\Sch}{\mathbf{S}}

\numberwithin{equation}{section}

\theoremstyle{plain}
\newtheorem{theorem}{\bf Theorem}[section]
\newtheorem*{theorem*}{Theorem 1.1$'$}
\newtheorem{lemma}[theorem]{\bf Lemma}

\theoremstyle{definition}

\theoremstyle{remark}
\newtheorem*{remark*}{\bf Remark}
\newtheorem{remark}[theorem]{\bf Remark}

\newcommand{\wt}{\widetilde}
\newcommand{\eps}{\varepsilon}

\newcommand{\loc}{\mathrm{loc}}
\newcommand{\1}{\mathbbm{1}}

\newcommand{\bb}{{\mathbf{b}}}
\newcommand{\bh}{{\mathbf{h}}}
\newcommand{\bc}{{\mathbf{c}}}
\newcommand{\bg}{{\mathbf{g}}}

\newcommand{\fb}{{\mathfrak{b}}}
\newcommand{\fs}{{\mathfrak{s}}}

\newcommand{\bu}{{\mathbf{u}}}

\newcommand{\Gank}{\Gamma} 
\newcommand{\bGank}{{\mathbf{\Gamma}}}

\begin{document}

\title[Asymptotics of eigenvalues of Hankel operators]{Asymptotic behaviour 
of eigenvalues of  Hankel operators}

\author{Alexander Pushnitski}
\address{Department of Mathematics, King's College London, Strand, London, WC2R~2LS, U.K.}
\email{alexander.pushnitski@kcl.ac.uk}

\author{Dmitri Yafaev}
\address{Department of Mathematics, University of Rennes-1,
Campus Beaulieu, 35042, Rennes, France}
\email{yafaev@univ-rennes1.fr}



\begin{abstract} 
We consider compact Hankel operators realized in $ \ell^2(\bbZ_+)$
as infinite matrices $\Gank$ with matrix elements $h(j+k)$. Roughly speaking, we show that if
$h(j)\sim (b_{1}+ (-1)^j b_{-1}) j^{-1}(\log j )^{-\alpha}$ as $j\to \infty$
 for some $\alpha>0$, then the eigenvalues of $\Gank$ satisfy 
 $\lambda_{n}^{\pm} (\Gank)\sim c^{\pm} n^{-\alpha}$  
as $n\to \infty$. The asymptotic coefficients $c^{\pm}$ 
are explicitly expressed in terms of  the asymptotic coefficients $b_{1} $ and $b_{-1}$.  
Similar results are obtained for Hankel operators $\bGank$ realized in $ L^2(\bbR_+)$ as integral operators with kernels $\bh(t+s)$. In this case the asymptotics  of eigenvalues   $\lambda_{n}^{\pm} (\bGank)$ are determined by the behaviour  of $\bh(t)$ as $t\to 0$ and as $t\to \infty$.
\end{abstract}

\date{8 December 2014}

\maketitle

\section{Introduction}\label{sec.a}

\subsection{Overview}\label{sec.a1}
For a   sequence $\{h(j)\}_{j=0}^\infty$     of complex numbers,
a Hankel operator $\Gank(h)$ in the space $ \ell^2(\bbZ_+)$ is formally defined as the 
``infinite matrix'' $\{h(j+k)\}_{j,k=0}^\infty$, that is,
\begin{equation}
(\Gank(h) u) (j)= \sum_{k=0}^\infty h(j+k) u (k), \quad u= (u (0), u (1), \ldots).
\label{eq:a5}
\end{equation}
We also consider integral Hankel operators $\bGank(\bh)$ in the space  $L^2(\bbR_+)$
($\bbR_+=(0,\infty)$),
formally defined by
\begin{equation}
(\bGank(\bh)\bu)(t)=\int_0^\infty \bh(t+s)\bu (s)ds,
\label{a5}
\end{equation}
where $\bh\in L^1_\loc(\bbR_+)$; 
this function is called the \emph{kernel} of the Hankel operator $\bGank(\bh)$. 
Under the assumptions below the operators $\Gank(h)$ and $\bGank(\bh)$ are bounded.
We will refer to the theory of Hankel operators in $ \ell^2(\bbZ_+)$ as  to the ``discrete case''
and to the one of the integral Hankel operators in $L^2(\bbR_+)$  as to the ``continuous case'';
objects related to the continuous case will be denoted by boldface symbols. 
Of course the operator $\Gank(h)$ (resp. $\bGank(\bh)$)  is self-adjoint 
if and only if the sequence $\{h(j)\}$ (resp. the function $\bh(t)$) is real valued. 
Background information on the theory of Hankel operators can be found in the book \cite{Peller} by  V.~Peller.

In this paper, we are interested in compact self-adjoint Hankel operators. 
Sharp estimates of eigenvalues of Hankel operators (and, more generally, of singular values in the non-self-adjoint case) are very well known.
At the same time, there are practically no results on the asymptotic behaviour
of these eigenvalues. 
The only exceptions known to us are the papers \cite{Widom, Yafaev2},
which will be  discussed below. 
This state of affairs is in a sharp contrast with the case of differential operators, 
where the Weyl type asymptotics of eigenvalues is established in a large variety of situations. 
Our goal here is to fill in this gap by 
describing a class of Hankel operators where the eigenvalue asymptotics (in the power  scale)
can be found explicitly.

Our approach relies on the following three ingredients:
\begin{enumerate}[(i)]
\item
A  result of \cite{Yafaev3} which establishes the unitary equivalence of Hankel operators  
to pseudodifferential operators ($\Psi$DO)  in $L^2 (\bbR)$ of a certain special class.
\item
Standard Weyl type  spectral asymptotics for  the corresponding $\Psi$DO, obtained by Birman and Solomyak in  \cite{BS4,BS3}.
\item
Estimates for singular values of Hankel operators from \cite{0}
(based on earlier results by Peller).
\end{enumerate}

In general, the study of eigenvalue asymptotics for any class of operators
involves two steps: 
construction of an appropriate model problem where the eigenvalue asymptotics
can be determined more or less explicitly, and using eigenvalue estimates  (or variational methods)
to extend the asymptotics   to a wider class of operators. 

As mentioned above, the relevant estimates in a convenient form 
were prepared in our previous paper \cite{0}.
The most important novel feature of this work is the construction of 
the appropriate model Hankel operators. 
In order to construct model Hankel operators, we proceed in two steps. 
Given a Hankel operator $\Gamma$,
first we construct a suitable  $\Psi$DO $\Psi_*$ of a negative order such that 
the spectral asymptotics of $\Psi_*$ can be established (see item (ii) above).
Then we use the unitary equivalence (see item (i) above) to map
$\Psi_*$ into a Hankel operator $\Gamma_*$  with the same spectrum.   For a ``correct" choice of $\Psi_*$, the Hankel operators  $\Gamma$  and $\Gamma_*$ are close to each other, and so they have the same leading terms of eigenvalue asymptotics. In more detail, our approach is outlined  in Sections~3.1 and 4.1 for the operators \eqref{a5} and
\eqref{eq:a5}, respectively.

\subsection{Discrete case}
 Let $\{\lambda_n^+(\Gank)\}_{n=1}^\infty$ be the non-increasing sequence
of positive eigenvalues of a compact 
self-adjoint operator $\Gank$ (with multiplicities taken into account), and let $\lambda_n^-(\Gank)=\lambda_n^+(-\Gank)$. We define also the eigenvalue counting 
function
\begin{equation}
n_\pm(\eps;\Gank)
=
\# \{n: \lambda_n^\pm(\Gank)>\varepsilon\}, 
\quad \eps>0.
\label{eq:CF}
\end{equation}

We start our discussion from the discrete case. 
In order to motivate our main result, 
let us consider   the sequence
\begin{equation}
h(j)=\frac1{(j+1)^\gamma}, 
\quad 
\gamma\geq1.
\label{a1d}
\end{equation}
If $\gamma=1$ then  the corresponding Hankel operator $\Gank(h)$, known as the Hilbert matrix,  
is bounded  (but not compact). 
From here by a simple argument one obtains
\begin{align}
h(j)=O(j^{-1}), \quad j\to\infty
\quad &\Rightarrow \quad 
\Gank(h) \text{ is bounded,}
\label{a1ff}
\\
h(j)=o(j^{-1}), \quad j\to\infty
\quad &\Rightarrow \quad 
\Gank(h) \text{ is compact.}
\label{a1f}
\end{align}
Roughly speaking, one expects that a faster rate of convergence of the sequence $h(j)$ to zero
as $j\to\infty$
results in a faster convergence of the eigenvalues $\lambda^\pm_n(\Gank(h))$ to zero
as $n\to\infty$. 
Indeed, there is a  deep result of H.~Widom who showed  in \cite{Widom} that for $\gamma>1$ 
the Hankel operator corresponding to the sequence \eqref{a1d} is non-negative   and its eigenvalues converge to zero \emph{exponentially} fast:
$$
\lambda_n^+(\Gank(h))=\exp(-\pi \sqrt{2\gamma n}+o(\sqrt{n})), \quad n\to\infty.
$$

Our goal is to study the case intermediate between 
$\gamma=1$ and $\gamma> 1$, when $h(j)$ behaves  as  $j^{-1} ( \log j)^{-\alpha}$ with some $\alpha>0$  for large $j$.
To give the flavour of our main result, 
first we state it in a particular case;
the full statement is given in Theorem~\ref{cr.a3} below. 
We use the notation $x_\pm=\max\{0,\pm x\}$;  $B(\cdot,\cdot)$ is the 
standard Beta function, 
\begin{equation}
B(a,b)=\int_1^\infty (t-1)^{a-1}t^{-a-b}dt
=
\frac{\Gamma(a)\Gamma(b)}{\Gamma(a+b)}
\label{a16}
\end{equation}
(of course, the symbols $\Gamma$ in the r.h.s. of \eqref{a16} stand for the Gamma function
rather than for Hankel operators). 
Put
\begin{equation}
v(\alpha):=2^{-\alpha}
\pi^{1-2\alpha}
B(\tfrac1{2\alpha},\tfrac12)^{\alpha};
\label{a4}
\end{equation}
in particular, $v(1)=2^{-\alpha}$.
In what follows, $\log$ denotes the natural (base $e$) logarithm. 
 
\begin{theorem}\label{thm.aa1}
Let $\alpha>0$, $b_1,b_{-1}\in\bbR$, and let  
\begin{equation}
h(j)=(b_1+(-1)^j b_{-1}) j^{-1}(\log j)^{-\alpha}, \quad j\geq2;
\label{a2}
\end{equation}
the choice of $h(0)$ and $h(1)$ 
\emph{(\emph{or of any finite number of $h(j)$})} is not important. 
Then   the eigenvalues of the corresponding Hankel operator $\Gank(h)$   
have the asymptotic behaviour 
\begin{equation}
\lambda_n^\pm(\Gank(h))
=
c^\pm n^{-\alpha}+o(n^{-\alpha}), 
\quad
c^\pm=v(\alpha)\bigl((b_1)^{1/\alpha}_\pm+(b_{-1})^{1/\alpha}_\pm\bigr)^\alpha,
\label{a1c}
\end{equation}
as $n\to\infty$.
\end{theorem}

Theorem~\ref{thm.aa1} shows that between the cases $\gamma=1$ and $\gamma>1$ in 
\eqref{a1d} there is a whole scale (the logarithmic scale)
of sequences $h  (j) $ such that the eigenvalues of $\Gank(h)$ have 
power spectral asymptotics. 

The full version of this result (which is stated below as Theorem~\ref{cr.a3}) also allows for an error term in 
\eqref{a2}.

\subsection{Discussion}
(1) Theorem~\ref{thm.aa1} is consistent with the Hilbert-Schmidt conditions for $\Gank(h)$. 
Indeed, in the self-adjoint case we have
\begin{equation}
\sum_{j=0}^\infty |h(j)|^2 (j+1)
=
\sum_{n=1}^\infty \bigl(\lambda_n^+(\Gank)^2+\lambda_n^-(\Gank)^2\bigr), 
\quad
\Gank=\Gank(h),
\label{eq:HS}
\end{equation}
and the operator $\Gank$ belongs to the Hilbert-Schmidt class if and only if 
the series  in the l.h.s. of \eqref{eq:HS} converges. 
This is true  if
\begin{equation}
h(j)  = O(j^{-1} (\log j)^{-\alpha}), \quad j\to\infty,
\label{eq:HS1}
\end{equation}
for some $\alpha> 1/2$. On the other hand, the series in the r.h.s. of \eqref{eq:HS} converges if
\begin{equation}
\lambda_n^\pm(\Gank(h))
= O(n^{-\alpha}),   \quad  n\to\infty, 
\label{a1cx}
\end{equation}
with some $\alpha>1/2$. 
This agrees with Theorem~\ref{thm.aa1}.

(2)
It is shown in \cite{0} that, for $0<\alpha<1/2$,
condition \eqref{eq:HS1} implies \eqref{a1cx}. 
This result remains true also for $ \alpha\geq 1/2$ if additionally one imposes some conditions on the iterated differences of the sequence $h(j)$; 
see Section~\ref{sec.a20} for the precise statement. 
This is also consistent with Theorem~\ref{thm.aa1}.
In particular,  Theorem~\ref{thm.aa1} shows that the above result of \cite{0}  is sharp.

(3) 
According to formula    \eqref{a1c} the sequences 
\begin{equation}
h_1(j)= j^{-1}(\log j)^{-\alpha}
\quad\text{ and }\quad
h_{-1}(j)=(-1)^j j^{-1}(\log j)^{-\alpha}
\label{a4b}
\end{equation}
yield the same spectral  asymptotics. 
This fact has a simple explanation:
 if $h_1(j)$ and $h_{-1} (j)$ are any two sequences such that
$h_{-1}(j)=(-1)^j h_1(j)$, 
then we have
\begin{equation}
\Gank(h_{-1})=F^* \Gank(h_1) F 
\quad\text{ where }\quad (F u)(j)=(-1)^j u(j), \quad j\geq0. 
\label{a4a}
\end{equation}
Thus, the operators $\Gank(h_1)$ and $\Gank(h_{-1})$ are unitarily equivalent 
and so they have the same eigenvalues.

(4) 
Let us discuss the structure of the formula    \eqref{a1c}  for the asymptotic coefficient $c^\pm$. 
  In terms of the counting function, it can be equivalently rewritten as
$$
 n_\pm(\varepsilon; \Gank(h))
=
 v(\alpha)^{1/\alpha}\bigl((b_1)^{1/\alpha}_\pm+(b_{-1})^{1/\alpha}_\pm\bigr)  \varepsilon^{-1/\alpha}(1+o(1)), \quad \varepsilon\to 0.
 $$
It follows that (using notation \eqref{a4b})
\begin{equation}
n_\pm(\eps;\Gank(h))
=
n_\pm(\eps; b_1 \Gank(h_1))+n_\pm(\eps; b_{-1}\Gank(h_{-1}))+o(\eps^{-1/\alpha}), 
\quad
\eps\to0.
\label{a4d}
\end{equation}
Roughly speaking, this means that the operator $\Gank(h)$ is in some sense asymptotically equivalent 
to the orthogonal sum $b_1\Gank(h_1)\oplus b_{-1}\Gank(h_{-1})$. The ``asymptotic orthogonality'' of $\Gank(h_1)$ and $\Gank(h_{-1})$ may look mysterious here,
but it will become clearer in the course of constructing the model operators, see Remark~\ref{rmk.b2}.

\subsection{Continuous case}\label{sec.a1a}
In the discrete case, the spectral asymptotics of $\Gank(h)$ is determined by the asymptotic
behaviour of the sequence $h(j)$ as $j\to\infty$. 
In the continuous case, the behaviour of the kernel $\bh(t)$ for $t\to0$ and for $t\to\infty$
as well as the local singularities of $\bh (t)$  
contribute to the asymptotics of the eigenvalues   of the Hankel operator $\bGank(\bh)$.
In this paper, we consider the kernels without local singularites for $t>0$.

It is well known that the Carleman operator, corresponding to the kernel
$$
\bh(t)=1/t,
$$
is bounded, but not compact.
From here, similarly to \eqref{a1ff}, \eqref{a1f}, one easily obtains
$$
\bh(t)=O(1/t) \text{ as $t\to0$ and as $t\to\infty$ } 
\quad \Rightarrow \quad
\bGank(\bh) \text{ is bounded},
$$
$$
\bh(t)=o(1/t) \text{ as $t\to0$ and as $t\to\infty$ } 
\quad \Rightarrow \quad
\bGank(\bh) \text{ is compact.}
$$
All our kernels will satisfy the latter condition. 
In the same way as for the ``discrete"  Hankel operator, first we give the result in an important particular
case; the full statement is given below as Theorem~\ref{thm.d1}.
\begin{theorem}\label{thm.aa2}
Let $\bh$ be a real valued function in $C^\infty(\bbR_+)$ such that 
$$
\bh(t)=
\begin{cases}
\bb_0 t^{-1} (\log(1/t))^{-\alpha}, & t<1/a, 
\\
\bb_\infty t^{-1} (\log t)^{-\alpha}, & t>a, 
\end{cases}
$$
with some $\alpha>0$, $a>1$  and $\bb_0,\bb_\infty\in\bbR$. 
Then
\begin{equation}
\lambda_n^\pm(\bGank(\bh))
=
\bc^\pm n^{-\alpha}
+
o(n^{-\alpha}),
\quad
\bc^\pm
=
v(\alpha)\bigl((\bb_0)_\pm^{1/\alpha}+(\bb_\infty)_\pm^{1/\alpha}\bigr)^\alpha,
\label{d3}
\end{equation}
as $n\to\infty$,
where $v(\alpha)$ is given by \eqref{a4}. 
\end{theorem}

In order to discuss the continuous case further, it is convenient to introduce some notation.
Let us
fix two cut-off functions  $\chi_0,\chi_\infty\in C^\infty(\bbR_+)$ 
such that  
\begin{equation}
\chi_0(t)=
\begin{cases}
1& \text{for $t\leq1/4$,}
\\
0& \text{for  $t\geq1/2$,}
\end{cases}
\qquad
\chi_\infty(t)=
\begin{cases}
0& \text{for $t\leq2$,}
\\
1& \text{for $t\geq4$,}
\end{cases}
\label{a7a}
\end{equation}
and define the model kernels 
\begin{equation}
\bh_0(t)=t^{-1}\abs{\log t}^{-\alpha}\chi_0(t),
\quad
\bh_\infty(t)=t^{-1}\abs{\log t}^{-\alpha}\chi_\infty(t),
\quad t>0.
\label{a7b}
\end{equation}
Observe that the hypothesis of Theorem~\ref{thm.aa2} is equivalent to the 
representation 
$$
\bh=\bb_0\bh_0+\bb_\infty \bh_\infty+\bg,
$$
where $\bg(t)$ is a smooth function that vanishes both for small and for large $t$.
We will see that the contribution of $\bGank(\bg)$ to the eigenvalue asymptotics
is negligible. 
Next, since the singularity of $\bh_0$ is located at zero and the singularity of $\bh_\infty$
is located at infinity, it is not surprising that the
operators $\bGank(\bh_0)$ and $\bGank(\bh_\infty)$ are ``asymptotically orthogonal'',
i.e.\ that, similarly to \eqref{a4d}, we have
$$
n_\pm(\eps;\bGank(h))
=
n_\pm(\eps; \bb_0\bGank(\bh_0))+n_\pm(\eps; \bb_\infty\bGank(\bh_\infty))+o(\eps^{-1/\alpha}), 
\quad
\eps\to 0.
$$
This explains the structure of formula \eqref{d3} for the asymptotic coefficient
$\bc^\pm$.

Compact Hankel operators $\bGank(\bh)$ with  kernels $\bh(t)$  
that have a singularity at a single point $t=t_0>0$ were  considered in 
\cite[Section 3]{GLP} and in \cite[Section 6]{Yafaev2}. 
In this case the eigenvalues $\lambda_{n}^\pm$ also have
  the power  asymptotics as $n\to \infty$, but the leading  terms of $\lambda_{n}^+$ and $\lambda_{n}^-$ are the same. 
In  the present paper we  consider locally regular kernels with a slow decay as $t\to \infty$ and singular at $t=0$. Thus the results as well as the methods  of 
\cite{GLP,Yafaev2} and those of the current paper are independent   and complement each other.

The results of this paper in the discrete and continuous cases are not fully independent of each other.
There are different ways of relating 
the operators  $\Gank(h)$ and $\bGank(\bh)$, e.g. via the Laguerre transform,
or by linking the corresponding symbols via a conformal change of variable
(see e.g. \cite[Section~1.8]{Peller}).  
Either of these methods shows that the  singularity (see \eqref{a7b}) of the kernel $\bh(t) $ at $t=0$ (resp., at $t=\infty$) corresponds to $h(j)$ 
 with asymptotics $j^{-1} (\log j)^{-\alpha}$ (resp., $(-1)^j j^{-1} (\log j)^{-\alpha}$) as $j\to \infty$.
However, technically it turns out to be more convenient to give two independent arguments
for the discrete and continuous cases. 
We also note that some features of the problem are more transparent in the discrete
case, while others are in the continuous case.

\subsection{The structure of the paper}
As already mentioned, our paper relies on a synthesis of various results. They are collected
in Section~\ref{sec.a2}. 
It is convenient to start the proofs with the continuous case. 
Thus, in Section~\ref{sec.d} we state and prove our main result in the continuous case,
and in Section~\ref{sec.e} we return to the discrete case. 
Finally,   a proof of an   assertion for $\Psi$DO in $L^2({\bbR})$ supplementing \cite{BS4} is given in the Appendix.

\section{Preliminaries}\label{sec.a2}
Here we discuss one by one the   three key ingredients 
of our approach mentioned in Section~1.1.
 
\subsection{Reduction to $\Psi$DO}\label{sec.a2a}
 
Let  $X$ and $D$ be self-adjoint operators in $L^2(\bbR)$ defined by 
\begin{equation}
(Xf)(x)=xf(x), \quad (Df)(x)=-if'(x).
\label{a8a}
\end{equation}
Denote
\begin{equation}
\fb(x):=(\pi/\cosh(\pi x))^{1/2}, \quad x\in\bbR.
\label{eq:Xx}
\end{equation}
This standard function  plays a distinguished  role in the theory of Hankel operators.
 
\begin{theorem}\label{thm.d2}\cite[Theorem 4.3]{Yafaev3}
Let $\sigma\in L^\infty(\bbR_+)$, and let $\bh$ be the Laplace 
transform of $\sigma :$ 
$$
\bh(t) =\int_0^\infty e^{-\lambda t}\sigma(\lambda)d\lambda = :  (\calL \sigma)(t),
\quad t>0.
$$
Then the Hankel operator $\bGank(\bh)$ in $L^2(\bbR_+)$ is unitarily equivalent 
to the $\Psi$DO 
\begin{equation}
\Psi=\fb(X)\fs (D)\fb(X)   
\label{a9b}
\end{equation}
in $L^2(\bbR) $ with
\begin{equation}
\fs(\xi)=\sigma(e^{-\xi}), 
\quad
\xi\in\bbR.
\label{eq:a9b}\end{equation}
\end{theorem}

The unitary equivalence of the operators $\bGank(\bh)$ and $\Psi$
is given essentially by the Mellin transform. In \cite{Yafaev2}, the function $\fs(\xi)$ is called the sign-function of the kernel $\bh (t)$ since it determines the sign of the corresponding Hankel operator $\bGank(\bh)$.

In the discrete case, the role of the Laplace transform of $\sigma (\lambda)$ is played by the sequence of moments of a function $\eta(\mu)$ defined on the interval $(-1,1)$.
Similarly to Theorem~\ref{thm.d2}, we have

\begin{theorem}\label{thm.e2}\cite[Theorem 7.7]{Yafaev3}
Let $\eta\in L^\infty(-1,1)$, 
and let $h(j)$ be the sequence of moments of $\eta$:
$$
h(j)=\int_{-1}^1 \eta(\mu)\mu^j d\mu, 
\quad
j=0,1,2,\dots.
$$
Then the Hankel operator $\Gank(h)$ in $\ell^2(\bbZ_+)$ is unitarily equivalent to the $\Psi$DO 
\eqref{a9b} in $L^2(\bbR) $ with
$$
\fs(\xi)=\eta\left(\frac{2e^{-\xi}-1}{2e^{-\xi}+1}\right), 
\quad \xi\in\bbR.
$$
\end{theorem}

\begin{remark}
A similar statement, but requiring $\eta\geq0$, was proven earlier by Widom in \cite{Widom}.
Widom establishes the unitary equivalence of $\Gank(h)$ to 
\begin{equation}
\fs(D)^{1/2}\fb(X)^2 \fs(D)^{1/2}.
\label{d5}
\end{equation}
Since for any bounded operator $T$, the non-zero parts of the operators $T^*T$ and $TT^*$ 
are unitarily equivalent, taking $T=\fs(D)^{1/2}\fb(X)$, we see that Widom's result
is essentially equivalent to  Theorem~\ref{thm.e2}, if $\fs\geq0$. 
We note that the study of spectral asymptotics of $\Gank(h)$ in \cite{Widom} 
also relies on the reduction to the $\Psi$DO \eqref{d5}.
\end{remark}

\subsection{Weyl asymptotics of $\Psi$DO}\label{sec.a2b}

 We need the following result.
 
\begin{theorem}\label{thm.b2}
Let $a\in C^\infty ({\bbR})$ be a  real-valued 
function such that
\begin{equation}
a(\xi)=
\begin{cases}
A(+\infty)\xi^{-\alpha}(1+o(1)), &
\xi\to\infty,
\\
A(-\infty)\abs{\xi}^{-\alpha}(1+o(1)), &
\xi\to-\infty,
\end{cases}
\label{b2}
\end{equation}
for some   $\alpha>0$ and some constants  $A(+\infty)$ and $A(-\infty)$.
Assume that $b(x)=\overline{b(x)}$ and
\begin{equation}
\abs{b(x)}\leq C\jap{x}^{-\rho}, 
\quad 
x\in\bbR,
\label{b1a}
\end{equation}
for some $\rho>\alpha/2$. 
Then for the pseudodifferential operator $\Psi=b(X)a(D)b(X)$ in $L^2(\bbR)$ one has 
\begin{equation}
\lambda_n^\pm(\Psi)
=
C^{\pm} n^{-\alpha}+o(n^{-\alpha}), 
\quad 
n\to\infty,
\label{b4}
\end{equation}
where the coefficients $C^\pm$ are given by 
\begin{equation}
C^{\pm}
=
(2\pi)^{-\alpha}
\bigl(
A(-\infty)_\pm^{1/\alpha}+A(+\infty)_\pm^{1/\alpha}
\bigr)^\alpha
\biggl(
\int_\bbR \abs{b(x)}^{2/\alpha}dx\biggr)^\alpha.
\label{b5}
\end{equation}
\end{theorem}

\begin{remark}
The asymptotic relations \eqref{b4}, \eqref{b5} can be equivalently rewritten 
in terms of the eigenvalue counting functions as
\begin{equation}
\lim_{\varepsilon\to0} \varepsilon^{1/\alpha} \#\{n: \lambda_n^\pm(\Psi)>\varepsilon\}
=
\frac1{2\pi}\lim_{\varepsilon\to0}\varepsilon^{1/\alpha}
\meas\{(x,\xi)\in\bbR^2: \pm a(\xi)b(x)^2>\varepsilon\}, 
\label{a12a}
\end{equation}
which is 
 the Weyl semiclassical formula. 
\end{remark}

\begin{remark}\label{rmk.b2}
Fix some $b$ satisfying the estimate \eqref{b1a}; 
let $\Psi_+$ correspond to some function $a$ with $A(+\infty)=1$, $A(-\infty)=0$ and 
let $\Psi_-$ correspond to the case $A(+\infty)=0$, $A(-\infty)=1$. Then, for a general
$\Psi$ as in  Theorem~\ref{thm.b2}, we can write
$$
\Psi=A(+\infty)\Psi_++A(-\infty)\Psi_-+\text{error term},
$$
and \eqref{b4}, \eqref{b5} mean that the operators $\Psi_+$, $\Psi_-$ are 
``asymptotically orthogonal'', i.e. 
$$
n_\pm(\eps;\Psi)
=
n_\pm(\eps; A(+\infty)\Psi_+)+n_\pm(\eps; A(-\infty)\Psi_-)+o(\eps^{-1/\alpha}), 
\quad
\eps\to+0.
$$
In the context of the Weyl formula \eqref{a12a}, this asymptotic orthogonality 
does not look very surprising as it corresponds to the symbols of the 
operators $\Psi_+$ and  $\Psi_-$ ``living'' in different parts of the phase space.  
\end{remark}

For compactly supported $b$, 
Theorem~\ref{thm.b2}    
was proven in \cite{BS4} where the multi-dimensional case was considered. 
Extension to arbitrary functions $b$ satisfying \eqref{b1a} is an easy application of
Cwikel type estimates for $\Psi$DO of the type $f(X)g(D)$; 
for completeness we give the proof in the Appendix. We also note that there was  an inessential restriction $\alpha\not\in{\bbZ}_{+}$ in \cite{BS4}. It 
  appeared   only because $\Psi$ was regarded in \cite{BS4} as an integral
operator rather than a $\Psi$DO.

Theorem~\ref{thm.b2} concerns a very special class of $\Psi$DO
with factorisable amplitudes. 
For general $\Psi$DO with amplitudes asymptotically homogeneous at 
infinity, Weyl type formula for the asymptotics of the  spectrum was obtained in \cite{BS3}.

\subsection{Spectral estimates}\label{sec.a20}

Let us start with the discrete case when Hankel operators  are defined by formula  
\eqref{eq:a5}  in the space ${\ell}^2({\bbZ}_{+})$. 
Now we do not assume that the operators $\Gank$ are self-adjoint. 
We denote by $\{s_n(\Gank)\}_{n=1}^\infty$ the non-increasing sequence
of singular values of $\Gank$, i.e. $s_n(\Gank)=\lambda_n^+(\sqrt{\Gank^*\Gank})$. 

Here we discuss spectral estimates for Hankel operators   corresponding to 
the sequences $g (j)$ that satisfy 
$$
g(j)=o(j^{-1}(\log j)^{-\alpha}), \quad j\to\infty,
$$
for some $\alpha>0$.  
We   also need some  assumptions on iterated differences $g^{(m)} (j)$. 
These are the sequences defined iteratively by setting $g^{(0)} (j)=g  (j)$ 
and 
$$
g^{(m)}(j)=g^{(m-1)}(j+1)-g^{(m-1)}(j), \quad j\geq0.
$$
Let
\begin{equation}
M(\alpha)=
\begin{cases}
[\alpha]+1& \text{ if } \alpha\geq1/2,
\\
0, & \text{ if } \alpha<1/2,
\end{cases}
\label{c5}
\end{equation}
where $[\alpha]=\max\{ m\in\bbZ_{+}: m\leq \alpha\}$ is the integer part of $\alpha$. 
We impose conditions on $M(\alpha)$ iterated differences of $g(j)$.

\begin{theorem}\label{thm.a1}\cite{0}
Let $\alpha>0$, and let $M=M(\alpha)$ be as in \eqref{c5}.
Let $g (j)$ be a sequence of complex numbers that satisfies 
\begin{equation}
  g^{(m)}(j)
=
o(j^{-1-m}(\log j)^{-\alpha}), \quad j\to\infty,
\label{ca7}
\end{equation}
for all $m=0,\dots,M$. Then 
$$
s_n(\Gank(g))=o(n^{-\alpha}), \quad n\to\infty.
$$
\end{theorem}

\begin{remark}
\begin{enumerate} 
\item
If instead of \eqref{ca7}, we have
\begin{equation}
  g^{(m)}(j)
=
O (j^{-1-m}(\log j)^{-\alpha}), \quad j\to\infty,
\label{ca7x}
\end{equation}
then
$$
s_n(\Gank(g))=O(n^{-\alpha}), \quad n\to\infty.
$$
\item
For the sequence defined by $g(j)=j^{-1}(\log j)^{-\alpha}$, $j\geq2$, condition \eqref{ca7x} is satisfied for all $m$.
\item
For $\alpha\geq1/2$, our choice of $M(\alpha)$ is probably not optimal, but 
it is not far from being so. 
Example~4.7  in \cite{0} shows that for $\alpha \geq 2$ 
one cannot take $M(\alpha)= [\alpha]-2$ in this theorem.
\end{enumerate}
\end{remark}

Let us give
the analogue of Theorem~\ref{thm.a1} in the continuous case, 
that is, for the operators $\bGank=\bGank(\bg)$ 
defined by formula \eqref{a5} in the space $L^2 ({\bbR}_{+})$.
We use the notation $\jap{x}=(\abs{x}^2+1)^{1/2}$. 
Similarly to the discrete case, for $\alpha<1/2$ we only need an 
assumption on $\abs{\bg(t)}$; for $\alpha\geq1/2$ we also need
assumptions on the derivatives $\bg^{(m)}(t)$. 

\begin{theorem}\label{thm.a2}\cite{0}
Let $\alpha>0$ and  let $M=M(\alpha)$ be the integer given by \eqref{c5}.
Let $\bg$ be a complex valued function,
$\bg\in L^\infty_\loc(\bbR_+)$; if $\alpha\geq1/2$, suppose also 
that $\bg\in C^M(\bbR_+)$.
Assume that $\bg$ satisfies
$$
\bg^{(m)}(t)=o(t^{-1-m}\jap{\log t}^{-\alpha})
\quad 
\text{ as $t\to0$ and as $t\to\infty$.}
$$
Then 
$$
s_n(\bGank(\bg))=o(n^{-\alpha}), \quad n\to\infty.
$$
\end{theorem}

As in the discrete case, Theorem~\ref{thm.a2} remains true if   $o$ is replaced by  $O$.

Theorems~\ref{thm.a1} and \ref{thm.a2} will be used in combination with 
the following standard result (see e.g. \cite[Section 11.6]{BSbook})
in spectral perturbation theory, which asserts the stability of eigenvalue asymptotics.

\begin{lemma}\label{lma.b1}\cite[Section 11.6]{BSbook}
Let $A$ and $B$ be compact self-adjoint operators and let $\alpha>0$.
Suppose that, for both signs $``\pm"$,   
$$
\lambda^\pm_n(A)
=
C^\pm n^{-\alpha}+o(n^{-\alpha}),
\quad \text{ and }\quad
s_n(B)
=
o(n^{-\alpha}),
\quad
\quad n\to\infty.
$$
Then 
$$
\lambda^\pm_n(A+B)
=
C^\pm n^{-\alpha}+o(n^{-\alpha}),
\quad n\to\infty.
$$
\end{lemma}

\section{Continuous case}\label{sec.d}

\subsection{Statement of the main result}
Our main result in the continuous case is

\begin{theorem}\label{thm.d1}
Let $\alpha>0$, $\bb_0,\bb_\infty\in\bbR$, and let $M=M(\alpha)$ be as defined in \eqref{c5}.
Let $\bh$ be a real valued function in $L^\infty_\loc(\bbR_+)$; if $\alpha\geq1/2$, assume also 
that $\bh\in C^M(\bbR_+)$.
Suppose that 
\begin{align}
\biggl( \frac{d}{dt}\biggr)^m
\bigl(\bh(t)-\bb_0 t^{-1}(\log(1/t))^{-\alpha}\bigr)
&=
o(t^{-1-m}\jap{\log t}^{-\alpha}), \quad t\to0, 
\label{d1}
\\
\biggl( \frac{d}{dt}\biggr)^m
\bigl(\bh(t)-\bb_\infty t^{-1}(\log t)^{-\alpha} \bigr)
&=
o(t^{-1-m}\jap{\log t}^{-\alpha}), \quad t\to\infty
\label{d2}
\end{align}
for  all $m=0,1,\dots,M$. 
Then the eigenvalues of the corresponding Hankel operator  $\bGank(\bh)$ 
have the asymptotic behaviour
\begin{equation}
\lambda_n^\pm(\bGank(\bh))
=
\bc^\pm n^{-\alpha}
+
o(n^{-\alpha}),
\quad
\bc^\pm
=
v(\alpha)\bigl((\bb_0)_\pm^{1/\alpha}+(\bb_\infty)_\pm^{1/\alpha}\bigr)^\alpha,
\label{d3a}
\end{equation}
as $n\to\infty$,
where $v(\alpha)$ is given by \eqref{a4}.  
\end{theorem}

Of course, this includes Theorem~\ref{thm.aa2} as a particular case. 

Let us describe the plan   of the proof of Theorem~\ref{thm.d1}. The first and the most important step  
is to construct a \emph{model operator}. To  that end, we introduce an auxiliary explicit function
$\sigma_*(\lambda)$ such that 
its Laplace transform   $\bh_* (t)=(\calL \sigma_*) (t)$ has the same asymptotics for $t\to 0$ and $t\to\infty$ as the kernel  
$\bh (t)$. To be more precise, we check that the difference $\bh-\bh_{*}$ satisfies the assumptions of 
Theorem~\ref{thm.a2} (singular value estimates).  
Then we apply the abstract Lemma~\ref{lma.b1} to conclude 
that the eigenvalues of the Hankel operators $\bGank(\bh)$ and $\bGank(\bh_{*})$ 
have the same asymptotic behaviour.

  Next,   Theorem~\ref{thm.d2} (reduction to $\Psi$DO) implies that the model Hankel operator $\bGank(\bh_*)$ is unitarily equivalent 
  to the $\Psi$DO $\Psi_{*} = \fb(X) \fs_*(D)\fb(X)$
in $L^2(\bbR)$ where $\fb(x)$ is the function \eqref{eq:Xx} and
$\fs_{*}(\xi)=\sigma_*(e^{-\xi})$.  
Theorem~\ref{thm.b2} (Weyl spectral asymptotics of $\Psi$DO) 
allows us to find the spectral asymptotics of the operator $\Psi_{*}$ and hence of $\bGank(\bh_*)$.

\subsection{The model operator}

Let us define the  auxiliary function $\sigma_*(\lambda)$ by the formula
\begin{equation}
\sigma_*(\lambda)
=
\bb_\infty\abs{\log \lambda}^{-\alpha}\chi_0(\lambda)
+
\bb_0\abs{\log \lambda}^{-\alpha}\chi_\infty(\lambda), 
\quad 
\lambda>0,
\label{d6}
\end{equation}
where the smooth cut-off functions $\chi_0$ and $\chi_\infty$ are defined by \eqref{a7a}. 
Our model operator is the Hankel operator $\bGank(\bh_*)$ where $\bh_*=\calL \sigma_*$. 

\begin{lemma}\label{lma.c0}
The eigenvalues of the model Hankel operator $\bGank(\bh_*)$
obey the asymptotic relation
\begin{equation}
\lambda_n^\pm(\bGank(\bh_*))
=
\bc^\pm n^{-\alpha}+o(n^{-\alpha}), 
\quad 
n\to\infty,
\label{e11}
\end{equation}
where the coefficients $\bc^\pm$ are given by \eqref{d3a}. 
\end{lemma}

\begin{proof}
According to Theorem~\ref{thm.d2}, the Hankel operator $\bGank(\bh_*)$
is unitarily equivalent to the $\Psi$DO $\Psi_{*} = \fb(X) \fs_*(D)\fb(X)$
in $L^2(\bbR)$ where $\fb(x)$ is the standard function \eqref{eq:Xx} and
$$
\fs_{*}(\xi)
=
\sigma_*(e^{-\xi})
=
\bb_\infty\abs{\xi}^{-\alpha}\chi_0(e^{-\xi})
+
\bb_0\abs{\xi}^{-\alpha}\chi_\infty(e^{-\xi}), 
\quad \xi\in\bbR.
$$
In particular, we have
\begin{equation}
\lambda_n^\pm(\bGank(\bh_*))
=
\lambda_n^\pm(\Psi_*), \quad n\in\bbN .
\label{eq:d6}
\end{equation}
Obviously, the function $\fs_{*}$ belongs to 
$C^\infty ({\bbR})$ and has the asymptotic behaviour \eqref{b2} with $A(+\infty)= \bb_{\infty}$ 
and $A(-\infty)= \bb_0$. 
Therefore Theorem~\ref{thm.b2} (Weyl spectral asymptotics of $\Psi$DO) applies to the operator $\Psi_{*}$.
This yields the asymptotic formula 
\begin{equation}
\lambda_n^\pm(\Psi_*)
=
C^\pm
n^{-\alpha}+o(n^{-\alpha}), 
\quad 
n\to\infty,
\label{e11a}
\end{equation}
where 
\begin{equation}
C^\pm
=
(2\pi)^{-\alpha}
\bigl((\bb_0)_\pm^{1/\alpha}+(\bb_\infty)_\pm^{1/\alpha}\bigr)^\alpha
\biggl(\int_{-\infty}^\infty(\pi(\cosh(\pi x))^{-1})^{1/\alpha} dx\biggr)^\alpha.
\label{d10}
\end{equation}
Using the change of variables $y=(\cosh(\pi x))^2$,     
the integral representation \eqref{a16} for the Beta function and the definition \eqref{a4} of $v(\alpha)$, 
we can   rewrite 
the integral in \eqref{d10}  as
\begin{multline}
(2\pi)^{-\alpha}
\biggl(\int_{-\infty}^\infty(\pi(\cosh(\pi x))^{-1})^{1/\alpha} dx\biggr)^\alpha
=
\pi^{1-\alpha}
\biggl(\int_{0}^\infty(\cosh(\pi x))^{-1/\alpha} dx\biggr)^\alpha
\\
=
2^{-\alpha}\pi^{1-2\alpha}
\biggl(\int_1^\infty y^{-\tfrac1{2\alpha}-\tfrac12}(y-1)^{-\tfrac12}dy\biggr)^\alpha
=
v(\alpha).
\label{d10a}
\end{multline}
Thus, $C^\pm=\bc^\pm$, where $\bc^\pm$ are the coefficients in \eqref{d3a}.
Combining \eqref{eq:d6} with \eqref{e11a}, we obtain \eqref{e11}.
\end{proof}

\subsection{Laplace transforms of functions with logarithmic singularities}

Let $\sigma_*(\lambda)$ be given by formula
\eqref{d6},   and  let $\bh_*=\calL \sigma_*$. 
Here we find the asymptotics of  the function $\bh_{*} (t)$  as $t\to \infty$ and $t\to 0$. To that end,
we need some  elementary technical statements about the Laplace transforms of 
functions with logarithmic singularities at $\lambda=0$ and $\lambda= \infty$. 
The results below are well known; see, e.g.,  Lemmas~3 and 4 in \cite{Erdelyi}. 
However, for completeness we give simple straightforward proofs.

\begin{lemma}\label{L}
Let
$$
I_m(t)= \int_{0}^c (-\log \lambda)^{-\alpha}\lambda^{m} e^{-\lambda t}d\lambda, 
$$
where $\alpha>0$,  $m\in\bbZ_+$ and $c\in (0,1)$. 
Then 
\begin{equation}
I_m(t)=m!\, t^{-1-m} |\log t|^{-\alpha} (1+O(\abs{\log t}^{-1})),
\label{eq:L2}
\end{equation}
as $t\to\infty$.
\end{lemma}
\begin{proof}
Let us split  $I_m(t)$ into the integrals
over $(0,t^{-1/2})$ and over $(t^{-1/2}, c)$. 
Due to the factor $e^{-\lambda t}$ the second integral 
can be estimated as
\begin{equation}
\int_{t^{-1/2}}^{c}
  (-\log\lambda)^{-\alpha} \lambda^m e^{-\lambda t}d\lambda
=
O(t^{-N}), \quad \forall N>0.
\label{A3a}
\end{equation}
Thus, it suffices to consider the integral over $(0,t^{-1/2})$. 
Making the change of variables $x=\lambda t$, we see that
$$
J (t):=
\int_0^{t^{-1/2}}
  (-\log\lambda)^{-\alpha} \lambda^m e^{-\lambda t}d\lambda
=
t^{-1-m}(\log t )^{-\alpha}\int_0^{t^{1/2}}  
  \bigl(1- \tfrac{\log x}{\log t}\bigr)^{-\alpha} x^m e^{-x}d x.
$$
Since $u= - \frac{\log x}{\log t} \geq -1/2$ for $x\leq t^{1/2}$, 
we can use the estimate  
\begin{equation}
\Abs{ (1+u)^{-\alpha} -1}
\leq 
C\abs{u}, 
\quad 
u\geq-1/2.
\label{eq:L}
\end{equation}
Thus we see that
\begin{equation}
J (t)=t^{-1-m}  (\log t )^{-\alpha}
\int_0^{t^{1/2}}  x^m  e^{-x}d x+R(t),
\label{eq:L4}
\end{equation} 
where the remainder $R(t)$ is estimated by
\begin{equation}
t^{-1-m}(\log t )^{-\alpha-1}\int_0^{\infty}    |\log x| x^m  e^{-x}d x.
\label{eq:L5}
\end{equation}
The integral in \eqref{eq:L4} can be extended to ${\bbR}_{+}$ and then calculated in terms of the Gamma function.  
The arising error decays   faster than any power of $t^{-1}$ as $t\to\infty$. 
Putting together \eqref{A3a} and \eqref{eq:L4}, we get \eqref{eq:L2}.
\end{proof}
  
Let us now state the assertion dual to Lemma~\ref{L}. Its proof is almost the same as that of Lemma~\ref{L}. 
  
\begin{lemma}\label{M}
Let
 \begin{equation}
I_m(t)= 
\int_{c}^\infty (\log \lambda)^{-\alpha}\lambda^{m} e^{-\lambda t}d\lambda
\label{eq:M1}
\end{equation}
where $\alpha>0$,  $m\in\bbZ_+$ and $c>1$. Then $I_m(t)$ has the asymptotic behaviour 
\eqref{eq:L2} as $t\to 0$.
\end{lemma}
   
 \begin{proof}
Now we    split \eqref{eq:M1} into the integrals
over $(c,t^{-1/2})$ and over $(t^{-1/2}, \infty)$. 
The first integral can be estimated by 
$C t^{-(m+1)/2}\abs{\log t}^{-\alpha}$. 
In  the integral over $(t^{-1/2}, \infty)$, we make the change of variables $x=\lambda t$ which yields the integral
$$
\int_{t^{-1/2}}^\infty
 (\log\lambda)^{-\alpha} \lambda^m  e^{-\lambda t}d\lambda
=
t^{-1-m}|\log t |^{-\alpha}\int_{t^{1/2}}^\infty 
  \bigl(1+ \tfrac{\log x}{|\log t|}\bigr)^{-\alpha} x^m  e^{-x}d x
  =:J (t).
$$
Since $ u=  \frac{\log x}{|\log t|} \geq -1/2$ for $x\geq t^{1/2}$,   we can use \eqref{eq:L} again so that
\begin{equation}
J (t)=t^{-1-m}\abs{\log t}^{-\alpha}
\int_{t^{1/2}}^\infty x^m e^{-y}dy+R(t),
\label{eq:MZ2}
\end{equation}
where the remainder is estimated by \eqref{eq:L5}.
The integral in \eqref{eq:MZ2} can be extended to $\bbR_+$ with 
an arising error estimated by $C t^{(m+1)/2}$. 
Putting together the results obtained, we obtain the asymptotics 
\eqref{eq:L2} as $t\to0$ for the integral \eqref{eq:M1}. 
\end{proof}
  
Combining  Lemmas~\ref{L} and \ref{M} we easily obtain the following result.

\begin{lemma}\label{lma.d4}
Let the function $\sigma_*$ be given by \eqref{d6}, and let $\bh_{*}=\calL\sigma_*$ be 
 its Laplace transform. Then
$$
\bh_{*}=\bb_0 \bh_0+\bb_\infty \bh_\infty+\wt \bg,
$$
where the model kernels $\bh_0$, $\bh_\infty$ are defined by \eqref{a7b} and the error term $\wt \bg\in C^\infty(\bbR_+)$ satisfies the estimates
\begin{equation}
\abs{\wt \bg^{(m)}(t)}\leq C_m t^{-1-m}\jap{\log t}^{-\alpha-1}, \quad t>0,
\label{d8}
\end{equation}
for all integers $m\geq0$.
\end{lemma}

\begin{proof}
First consider the case $\bb_0=0$, $\bb_\infty=1$. 
We have
\begin{equation}
\wt \bg(t)
=
\int_0^\infty \abs{\log\lambda}^{-\alpha}\chi_0(\lambda)e^{-\lambda t}d\lambda
-
t^{-1}\abs{\log t}^{-\alpha}\chi_\infty(t), 
\quad
t>0.
\label{d8a}
\end{equation}
It is clear that $\wt \bg\in C^\infty(\bbR_+)$ and that 
for all $m\in\bbZ_+$ we have $\wt \bg^{(m)}(t)=O(1)$ as $t\to0$. 
Thus, it suffices to check the estimates
\eqref{d8} for $t\to \infty$. 
Let us split the integral in \eqref{d8a} into the sum of the integrals 
over $(0,1/4)$ and $(1/4,1)$. 
The integral over $(1/4,1)$ (along with all of its derivatives in $t$) decays exponentially fast
as $t\to\infty$.  Since $\chi_0(\lambda)=1$ for $\lambda\leq 1/4$,  we have
\begin{multline*}
\biggl(\frac{d}{dt}\biggr)^m
\biggl(
\int_0^{1/4}\abs{\log \lambda}^{-\alpha}  \chi_0(\lambda) e^{-\lambda t} d\lambda
-
t^{-1}(\log t)^{-\alpha}
\biggr)
\\
= (-1)^m  
\biggl(
\int_0^{1/4}\abs{\log \lambda}^{-\alpha}\lambda^m  e^{-\lambda t} d\lambda
-
m! t^{-1-m}(\log t)^{-\alpha}
\biggr)+
O(t^{-1-m}(\log t)^{-\alpha-1})
\end{multline*}
as $t\to\infty$.
Using Lemma~\ref{L}, we see that the first term in the right-hand side
is also $O(t^{-1-m}(\log t)^{-\alpha-1})$.

Similarly, in  the case $\bb_0=1$, $\bb_\infty=0$  we have
\begin{equation}
\wt \bg(t)
=
\int_0^\infty \abs{\log\lambda}^{-\alpha}\chi_\infty(\lambda)e^{-\lambda t}d\lambda
-
t^{-1}\abs{\log t}^{-\alpha}\chi_0(t), 
\quad
t>0.
\label{d8b}
\end{equation}
Evidently, $\wt \bg\in C^\infty(\bbR_+)$, and this function, along with all of its
derivatives in $t$, decays exponentially fast as $t\to\infty$. 
So one only needs to prove the estimates \eqref{d8} for $t\to0$. 
We split the integral in \eqref{d8b} into the sum of the integrals over $(0,4)$ and $(4,\infty)$. 
The integral over $(0,4)$ is a   function of $t$ bounded with all its derivatives as $t\to 0$.
Differentiating the integral over $(4,\infty)$  and
applying Lemma~\ref{M} to  it, we  complete the proof of \eqref{d8}.

Finally, the general case is a linear combination of the two cases considered above. 
 \end{proof}

\subsection{Proof of Theorem~\ref{thm.d1}}

By the hypothesis of the theorem, we have the representation 
$$
\bh=\bb_0\bh_0+\bb_\infty \bh_\infty+\bg,
$$
where $\bg$ satisfies the hypothesis of Theorem~\ref{thm.a2} (singular value estimates). 
As above, let the function $\sigma_*$ be given by \eqref{d6}, and let $\bh_{*}=\calL\sigma_*$ be 
 its Laplace transform.
Using Lemma~\ref{lma.d4}, we obtain that the difference 
$$
\bh-\bh_*=\bg-\wt \bg
$$
also satisfies the hypothesis of Theorem~\ref{thm.a2}. 
Thus, 
\begin{equation}
s_n(\bGank(\bh-\bh_*))=o(n^{-\alpha}), \quad n\to\infty.
\label{eq:We}
\end{equation}
Let us now apply the abstract Lemma~\ref{lma.b1} to $A=\bGank(\bh_*)$ 
and $B=\bGank(\bh-\bh_*)$.
Then the desired result for the operator $\bGank(\bh)=A+B$  follows from Lemma~\ref{lma.c0} and from \eqref{eq:We}. 
\qed

\subsection{Matrix valued kernels}\label{MV}
Let $N\in\bbN$, and let $\bh$ be an $N\times N$ matrix valued function 
on $(0,\infty)$. The  Hankel operator $\bGank(\bh)$ 
in the space $L^2(\bbR_+,\bbC^N)$ is defined by the same formula \eqref{a5}
as in the scalar case. 
Such operators   appear, for example, in applications to systems theory,
see, e.g. \cite{GLP}. 

Theorem~\ref{thm.d1} extends to the matrix case without 
difficulty. In this case, $\bh(t)$ is a Hermitian matrix for all $t>0$ 
(this ensures the self-adjointness of  $\bGank(\bh)$), and the coefficients
$\bb_0$, $\bb_\infty$ in \eqref{d1}, \eqref{d2} are also Hermitian matrices. 
Formula \eqref{d3a} for the asymptotic coefficients $\bc^\pm$ becomes
$$
\bc^\pm
=
v(\alpha)\Bigl(\Tr\bigl((\bb_0)^{1/\alpha}_\pm\bigr)+\Tr\bigl((\bb_\infty)_\pm^{1/\alpha}\bigr)\Bigr)^\alpha,
$$
where the matrices $(\bb_0)_\pm^{1/\alpha}$, $(\bb_\infty)_\pm^{1/\alpha}$ are defined in the sense of 
the standard functional calculus for Hermitian matrices. 

Let us comment on the proof of this statement.
Theorem~\ref{thm.a2} (singular value estimates) extends to the matrix   case trivially. 
Theorem~\ref{thm.d2} (reduction to $\Psi$DO) also extends to the case when $\sigma$ is a matrix 
valued function; in this case $\Psi$ is a $\Psi$DO acting on vector valued functions. 
This $\Psi$DO is given by the same formula \eqref{a9b} as in the scalar case with 
 the matrix valued 
  function  $\mathfrak s(\xi)$ and  the standard scalar valued 
  function  $\mathfrak b(x)$ defined by the same formulas \eqref{eq:Xx} and \eqref{eq:a9b} as before.
Finally, the extension of Theorem~\ref{thm.b2} (Weyl spectral asymptotics of $\Psi$DO) to the matrix valued case
is not quite trivial, but fortunately it has been proven in \cite{BS4} already 
in the matrix   case. 
Putting together these ingredients in the same way as in the scalar
case, one obtains the spectral asymptotics for the matrix valued kernels.

\section{Discrete case}\label{sec.e}

\subsection{Statement of the main result}
Below is our main result in the discrete case.
\begin{theorem}\label{cr.a3}
Let $\alpha>0$, $b_1, b_{-1} \in\bbR$, 
and let $h$ be a sequence of real numbers given by 
\begin{equation}
h(j) =(b_1+ (-1)^j b_{-1} )j^{-1}(\log j)^{-\alpha}+g_1(j)+(-1)^j g_{-1}(j), \quad j\geq2,
\label{a2a}
\end{equation}
where the error terms ${g}_{\pm 1}$ satisfy  conditions \eqref{ca7} 
for all $m=0,1,\dots,M(\alpha)$ 
$(M(\alpha)$ is defined in \eqref{c5}$)$. 
Then the eigenvalues of  the corresponding Hankel operator $\Gank(h)$ 
have the asymptotic behaviour
\begin{equation}
\lambda_n^\pm(\Gank(h))
=
c^\pm n^{-\alpha}+o(n^{-\alpha}), 
\quad
c^\pm=v(\alpha)\bigl((b_1)^{1/\alpha}_\pm+(b_{-1})^{1/\alpha}_\pm\bigr)^\alpha,
\label{e3}
\end{equation}
as $n\to\infty$, where $v(\alpha)$ is given by \eqref{a4}. 
\end{theorem}

Theorem~\ref{thm.aa1} is a particular case of the last theorem  corresponding to 
$g_1=g_{-1}=0$. 
Observe that if $g_{-1}$ satisfies  \eqref{ca7} with some $m>0$, then the sequence
$(-1)^j g_{-1}(j)$ does not necessarily satisfy the same condition. Thus, the two 
term remainder $g_1+(-1)^j g_{-1}$ in \eqref{a2a} in general does not reduce to
one term $g_1$.

Just as in the continuous case (see Subsection~\ref{MV}), one can consider 
Hankel operators in $\ell^2(\bbZ_+,\bbC^N)$ defined by sequences 
$\{h(j)\}_{j=0}^\infty$ of Hermitian $N\times N$ matrices. 
In this case, the asymptotic coefficients $b_{\pm 1}$ in \eqref{a2a} are also Hermitian  matrices, and formula \eqref{e3}  holds true 
with the asymptotic coefficient  
$$
c^\pm
=
v(\alpha)\Bigl(\Tr\bigl((b_1)^{1/\alpha}_\pm\bigr)+\Tr\bigl((b_{-1})_\pm^{1/\alpha}\bigr)\Bigr)^\alpha.
$$

Let  us describe the plan of the proof of Theorem~\ref{cr.a3}. We follow the same steps as in Section~\ref{sec.d}, 
but instead of the Laplace transform $\bh_*=\calL \sigma_*$ of the function $\sigma(\lambda)$, $\lambda>0$, we consider the sequence of moments 
\begin{equation}
h_*(j)=\int_{-1}^1 \eta_*(\mu)\mu^j d\mu, \quad j\geq0 ,
\label{e19}
\end{equation}
of some  explicit function $ \eta_*(\mu)$ of $\mu\in(-1,1)$. Our model operator is $\Gank(h_*)$. With our choice of $ \eta_*(\mu)$,  the difference $h-h_{*}$ satisfies the assumptions of Theorem~\ref{thm.a1} (singular value estimates).
Therefore the eigenvalues of the Hankel operators $\Gank(h)$ and $\Gank(h_{*})$ 
have the same asymptotic behaviour.
 Next,   Theorem~\ref{thm.e2} implies that the Hankel operator $\Gank(h_*)$ is unitarily equivalent to the $\Psi$DO $\Psi_{*} = \fb(X) \fs_*(D)\fb(X)$
in $L^2(\bbR)$ where $\fb(x)$ is the standard function \eqref{eq:Xx} and
\begin{equation}
\fs_{*}(\xi)=\eta_*\left(\frac{2e^{-\xi}-1}{2e^{-\xi}+1}\right), 
\quad \xi\in\bbR.
\label{eq:ss}
\end{equation}
 Theorem~\ref{thm.b2} (Weyl spectral asympotics for $\Psi$DO) 
 allows us to find  spectral asymptotics of the operators $\Psi_{*}$ and hence of $\Gank(h_*)$.

\subsection{The model operator}
Let us define the function $\eta_*(\mu)$ by
 the following explicit formula:
\begin{equation}
\eta_*(\mu)
=
\Abs{\log \frac{1+\mu}{2(1-\mu)}}^{-\alpha}
\left(
b_1\chi_\infty\Big(\frac{1+\mu}{2(1-\mu)}\Big)
+
b_{-1}\chi_0\Big(\frac{1+\mu}{2(1-\mu)}\Big)
\right), \quad \mu\in(-1,1),
\label{e8}
\end{equation}
where the smooth cut-off functions $\chi_\infty$ and $\chi_0$ are given by equalities \eqref{a7a}.
Note that the function $\eta_*$ belongs to the class $C^\infty (-1,1)$ and $\eta_*(\mu)\to 0$   logarithmically as $\mu\to 1$ and as $\mu\to -1$.

\begin{lemma}\label{lma.e0}
Let $h_* (j)$ be the sequence of moments  \eqref{e19} of the function \eqref{e8}.
Then the eigenvalues of the  Hankel operator $\Gank(h_*)$
obey the asymptotic relation
\begin{equation}
\lambda_n^\pm(\Gank(h_*))
=
c^\pm n^{-\alpha}+o(n^{-\alpha}), 
\quad 
n\to\infty,
\label{e18}
\end{equation}
where the coefficients $c^\pm$ are given by \eqref{e3}. 
\end{lemma}

\begin{proof}
Let us use Theorem~\ref{thm.e2} with $\eta=\eta_*$. 
We get that the corresponding Hankel operator  $\Gank(h_*)$ is unitarily equivalent to the $\Psi$DO
$
\Psi_{*}
=
\fb(X)\fs_{*}(D)\fb(X),
$
  where the functions $\fb(x)$ and $\fs_{*}(\xi)$ are given by formulas \eqref{eq:Xx} and \eqref{eq:ss}, respectively. 
By the definition \eqref{e8} of $\eta_*$, we have
$$
\fs_{*}(\xi)=\abs{\xi}^{-\alpha}(b_1\chi_\infty(e^{-\xi})+b_{-1}\chi_0(e^{-\xi})), 
\quad 
\xi\in\bbR.
$$
Applying Theorem~\ref{thm.b2} to the $\Psi$DO  $\Psi_{*}$, we obtain
$$
\lambda_n^\pm(\Gank(h_*))
=
\lambda_n^\pm(\Psi_*)
=
C^\pm n^{-\alpha}
+
o(n^{-\alpha}),
\quad
n\to\infty,
$$
where 
$$
C^\pm
=
(2\pi)^{-\alpha}
\bigl((b_1)_\pm^{1/\alpha}+(b_{-1})_\pm^{1/\alpha}\bigr)^\alpha
\biggl(\int_{-\infty}^\infty(\pi(\cosh(\pi x))^{-1})^{1/\alpha} dx\biggr)^\alpha.
$$
It follows from \eqref{d10a}  that  $C^\pm=c^\pm$, 
where the numbers $c^\pm$ are   given by  \eqref{e3}.
\end{proof}

\subsection{Moments of functions with logarithmic singularities}

 Our goal here is to obtain the asymptotics of the sequence $h_*$ 
of moments of the function $\eta_*$. We use again Lemma~\ref{L} but 
in order to replace the continuous parameter $t$ with the discrete one $j$, 
we need the following simple statement.

\begin{lemma}\label{lma.A2}
Let $m\in\bbZ_+$, and let  $\mathbf{g}\in C^m(\bbR_+)$ be a function that satisfies the estimate
$$
\abs{\bg^{(m)}(t)}\leq C t^{-1-m}(\log t)^{-\alpha}, \quad t\geq 2.
$$
Let $\{g(j)\}_{j=2}^\infty$ be a sequence defined by $g(j)=\bg(j)$, $j\geq2$. 
Then  
$$
\abs{g^{(m)}(j)}
\leq 
C j^{-1-m}(\log j)^{-\alpha}, 
\quad j\geq 2.
$$
\end{lemma}

\begin{proof}
It suffices to use the explicit formula
$$
g^{(m)}(j)
=
\int_0^1 dt_1 
\int_0^1 dt_2
\cdots
\int_0^1 dt_m \, 
\bg^{(m)}(j+t_1+\cdots+t_m),
$$
which can be checked by induction in $m$.
\end{proof}

The following assertion plays the same role here as  Lemma~\ref{lma.d4} played in the
previous Section.

\begin{lemma}\label{lma.e3}
Let the sequence $h_*$ be defined by \eqref{e19}, \eqref{e8}. 
Then $h_* (j)$ has the  asymptotics 
\begin{equation}
h_{*}(j) =(b_1+  (-1)^j b_{-1})j^{-1}(\log j)^{-\alpha}+ \wt g_{1}(j)+(-1)^j \wt g_{-1}(j), \quad j\geq2,
\label{eq:a2a}
\end{equation}
where the error terms $\wt g_{\pm 1}(j)$  satisfy the estimates
\begin{equation}
 \wt g_{\pm 1}^{(m)} (j)
=O(j^{-1-m}(\log j)^{-\alpha-1}), \quad j\to\infty,
\label{e10}
\end{equation}
for all $m=0,1,2,\dots$. 
\end{lemma}

\begin{proof} 
Let the function $\eta_{*} (\mu)$ be defined by \eqref{e8}.
If  $b_1=1$, $b_{-1}=0$, then   $\eta_{*} (\mu)=0$ for $\mu\leq 0$ and
setting $\mu=e^{-\lambda}$ in \eqref{e8}, we see that
$$
h_*(j) 
=
\int_0^1 
\Big(\log \frac{1+\mu}{2(1-\mu)}\Big)^{-\alpha}
\chi_\infty \Big(\frac{1+\mu}{2(1-\mu)}\Big)
\mu^j d\mu
 =
(\calL\sigma_1)(j)
$$
where
$$
\sigma_1(\lambda)
=
\Big(\log \frac{1+e^{-\lambda}}{2(1-e^{-\lambda})}\Big)^{-\alpha}
\chi_\infty \Big(\frac{1+e^{-\lambda}}{2(1-e^{-\lambda})}\Big)
e^{-\lambda}. 
$$
If  $b_1=0$, $b_{-1}=1$, then   $\eta_{*} (\mu)=0$ for $\mu\geq 0$ and
setting $\mu=-e^{-\lambda}$ in \eqref{e8}, we see that   $h_*(j)=(-1)^j (\calL\sigma_{-1})(j)$ where
$$
\sigma_{-1}(\lambda)
=
\Abs{\log \frac{1-e^{-\lambda}}{2(1+e^{-\lambda})}}^{-\alpha}
\chi_0 \Big(\frac{1-e^{-\lambda}}{2(1+e^{-\lambda})}\Big)
e^{-\lambda}.
$$
In both cases we have
$$
\sigma_{\pm 1}(\lambda)
=(-\log \lambda)^{-\alpha}
+O(\abs{\log\lambda}^{-\alpha-1}) 
$$
as $\lambda\to 0$.

Thus, we can write
\begin{equation}
(\calL\sigma_{\pm 1})(t)
=
\int_0^{1/2} (-\log \lambda)^{-\alpha} e^{-\lambda t}d\lambda
+
\int_0^\infty\wt{\sigma}_{\pm 1}(\lambda)e^{-\lambda t}d\lambda,
\label{eq:LL}\end{equation}
where the functions $\wt{\sigma}_{\pm 1}(\lambda)$ satisfy the estimate
$$
\abs{\wt{\sigma}_{\pm 1}(\lambda)}\leq C\jap{\log \lambda}^{-\alpha-1},
\quad \lambda>0,
$$
and vanish identically for large $\lambda$. 
Differentiating \eqref{eq:LL} and 
using Lemma~\ref{L},   we see that
\begin{equation}
(\calL\sigma_{\pm 1}) (t)
=
t^{-1}(\log t)^{-\alpha}\chi_\infty(t)
+
\bg_{\pm 1}(t),
\label{eq:LL1}\end{equation}
where the functions  $\bg_{\pm 1} (t)$ satisfy the estimates \eqref{d8} for $t\geq 2$.

 For arbitrary $b_1$, $b_{-1}$,  we have
$$
h_* (j)= b_{1}(\calL\sigma_{1})(j) + (-1)^j b_{-1}(\calL\sigma_{-1})(j) .
$$
According to \eqref{eq:LL1} this yields representation \eqref{eq:a2a} with $\wt{g}_{\pm 1} (j)= b_{\pm 1} \bg_{\pm 1}(j)$. Estimates \eqref{e10} follow from 
Lemma~\ref{lma.A2}. \end{proof}

\subsection{Proof of Theorem~\ref{cr.a3}}

Let $h$ be as in the hypothesis  of Theorem~\ref{cr.a3},
and let, as above, the sequence $h_*$ be defined by \eqref{e19}, \eqref{e8}. 
Comparing \eqref{a2a} and \eqref{eq:a2a}, we see that
\begin{equation}
h(j) =h_{*} (j)+ f_1(j)+(-1)^j f_{-1}(j) 
\label{eq:LL2}\end{equation}
where the error terms $f_{\pm 1}(j)= g_{\pm 1}(j)- \widetilde{g}_{\pm 1}(j)$ satisfy  the condition
$$
f_{\pm 1}^{(m)}(j)= o(j^{-1-m}  (\log j)^{-\alpha}), \quad j\to\infty,
$$
for all $m=0,1,\dots,M(\alpha)$. According to  Theorem~\ref{thm.a1} (singular value estimates) we have
$s_n(\Gank(f_{\pm 1}))=o(n^{-\alpha}), \quad n\to\infty$.

Put $ \wt f_{-1}(j)=(-1)^j f_{-1}(j) $; then by the unitary equivalence  \eqref{a4a},  
one has $s_n(\Gank(\wt f_{-1}))=s_n(\Gank(f_{-1}))$ and hence
\begin{equation}
s_n(\Gank(f_1 + \wt f_{-1}))=o(n^{-\alpha}), \quad n\to\infty.
\label{e14+}
\end{equation}
Put $A=\Gank(h_*)$ and $B=\Gank(f_1 + \wt f_{-1})$ so that,  by \eqref{eq:LL2}, $\Gank(h)=A+B$.
In view of  \eqref{e18}  and \eqref{e14+}, we can apply the abstract Lemma~\ref{lma.b1} to these operators. This
 yields the eigenvalue asymptotics \eqref{e3} for $\Gank(h)$. 
\qed

\appendix
\section{Theorem~\ref{thm.b2} for   $\Psi$DO in $L^2({\bbR})$ }

 Theorem~\ref{thm.b2} was proven in \cite{BS3}
for compactly supported $b$.  
We only have to extend it to general $b$ satisfying the estimate \eqref{b1a}.

\subsection{Schatten class estimates}\label{sec.b1}
For $p>0$, the weak Schatten class $\Sch_{p,\infty}$ consists of compact operators $A$ such that 
$$
\norm{A}_{\Sch_{p,\infty}}:=
\sup_{n\geq 1} n^{1/p}s_n(A)<\infty.
$$
The functional $\norm{\cdot}_{\Sch_{p,\infty}}$ is a quasinorm on $\Sch_{p,\infty}$. 
Recall that the eigenvalue counting functions of self-adjoint compact operators was defined by formula \eqref{eq:CF}.
It is convenient to introduce the following functionals
for a self-adjoint operator $A\in \Sch_{p,\infty}$:
$$
\Delta_p^\pm(A)=\limsup_{\varepsilon\to0} \varepsilon^p n_\pm(\varepsilon,A),
\quad
\delta_p^\pm(A)=\liminf_{\varepsilon\to0} \varepsilon^p n_\pm(\varepsilon,A).
$$
In applications, one usually has $\Delta^\pm_p(A)=\delta^\pm_p(A)$, 
but it is technically convenient to analyse the functionals
$\Delta^\pm_p$ and $\delta^\pm_p$ separately. 
Observe that  for $\alpha>0$ and $p=1/\alpha$ the   relations
$$
\lim_{n\to\infty }n^\alpha\lambda^\pm_n(A) = C^\pm  
\quad \text{ and }\quad
\Delta_{p}^\pm(A)=\delta_{p}^\pm(A)=(C^\pm)^p
$$
are equivalent.
The functionals $\Delta_{p}^\pm$, $\delta_{p}^\pm$
are continuous in $\Sch_{p,\infty}$. In fact, one has
(see, e.g., \cite{BSbook}, formulas (11.6.10), (11.6.14), (11.6.15))
\begin{align}
\abs{(\Delta_{p}^\pm(A_1))^{\frac{1}{1+p}}
-(\Delta_{p}^\pm(A_2))^{\frac{1}{1+p}}}
\leq
\norm{A_1-A_2}^{\frac{1}{1+p}}_{\Sch_{p,\infty}},
\label{C5}
\\
\abs{(\delta_{p}^\pm(A_1))^{\frac{1}{1+p}}
-(\delta_{p}^\pm(A_2))^{\frac{1}{1+p}}}
\leq
\norm{A_1-A_2}^{\frac{1}{1+p}}_{\Sch_{p,\infty}}.
\label{C6}
\end{align}

Let $X$, $D$ be the operators in $L^2(\bbR)$ defined by \eqref{a8a}.
We need an estimate for operators of the form $f(X)g(D)$ in weak Schatten classes. 
This estimate uses lattice function classes. 
For $f\in L^2_{\loc}(\bbR)$, one writes
$$
v_f(n)=\left(\int_{n}^{n+1} \abs{f(x)}^2 dx\right)^{1/2}, \quad n\in\bbZ,
$$
and for $p>0$ one defines the lattice classes
\begin{align*}
f\in \ell^p(L^2) &\quad \Leftrightarrow \quad v_f\in \ell^p,
\\
f\in \ell^{p,\infty}(L^2) &\quad \Leftrightarrow \quad v_f\in \ell^{p,\infty}.
\end{align*}
Recall that 
$$
v\in \ell^{p,\infty}
\quad \Leftrightarrow \quad
\norm{v}_{\ell^{p,\infty}}=\sup_{\varepsilon >0}\varepsilon (\#\{n\in\bbZ: \abs{v(n)}>\varepsilon\})^{1/p}<\infty.
$$

The following theorem is a combination of ideas of Cwikel \cite{C} 
(who had a version for $p\geq2$) and  
Birman-Solomyak \cite[Theorem 11.1]{BS5}. 
The case $1<p\leq2$ appeared in \cite{Simon}.
Theorem~\ref{thm.C1} in the  form given below and the proof can be found, e.g., in \cite{BKS}.

\begin{theorem}\cite{BKS}\label{thm.C1}
Let $0<p\leq2$,  $f\in \ell^p(L^2)$, $g\in \ell^{p,\infty}(L^2)$. 
Then 
$f(X)g(D)\in \Sch_{p,\infty}$
and 
$$
\norm{f(X)g(D)}_{\Sch_{p,\infty}}
\leq 
C_p 
\norm{f}_{\ell^p(L^2)}\norm{g}_{\ell^{p,\infty}(L^2)}.
$$
\end{theorem}
In fact, we need only a very special case of this theorem.

\subsection{Approximation arguments.}
Let us come back to Theorem~\ref{thm.b2}.
For $N\in\bbN$, let $\1_{(-N,N)}$ be the characteristic function of the interval $(-N,N)$. Set
$$
b_N(x)=b(x)\1_{(-N,N)}(x), 
\quad
\wt b_N(x)=b(x)-b_N(x), 
\quad
\Psi_N=b_N(X)a(D)b_N(X).
$$

\begin{lemma}\label{appr}
Under the assumptions of Theorem~$\ref{thm.b2}$ we have
\begin{equation}
\lim_{N\to \infty}\norm{\Psi-\Psi_N}_{\Sch_{p,\infty}}
= 0, \quad p=1/\alpha.
\label{b7}
\end{equation}
\end{lemma}

\begin{proof}
Since
$$
\Psi-\Psi_N
=
b(X)a(D)\wt b_N(X)+\wt b_N(X)a(D)b_N(X),
$$
we see that
$$
\norm{\Psi-\Psi_N}_{\Sch_{p,\infty}}
\leq
2\norm{b(X)a(D)\wt b_N(X)}_{\Sch_{p,\infty}}.
$$
Using the ``Schwarz inequality for classes $\Sch_{p,\infty}$''
(see e.g. \cite{BSbook}, (11.6.17)), we obtain 
$$
\norm{b(X)a(D)\wt b_N(X)}_{\Sch_{p,\infty}}
\leq
\norm{b(X)\abs{a(D)}^{1/2}}_{\Sch_{2p,\infty}}
\norm{\abs{a(D)}^{1/2}\wt b_N(X)}_{\Sch_{2p,\infty}}.
$$
Under our assumptions on $a$ and $b$, we have 
$a\in \ell^{p,\infty}(L^2)$, $b\in \ell^{2p}(L^2)$ and 
$\norm{\wt b_N}_{\ell^{2p}(L^2)}\to0$ as
$N\to\infty$.
Therefore, by Theorem~\ref{thm.C1}, we have $b(X)\abs{a(D)}^{1/2}\in\Sch_{2p,\infty}$ and 
$$
\norm{\abs{a(D)}^{1/2}\wt b_N(X)}_{\Sch_{2p,\infty}}
\leq
C
\norm{a}^{1/2}_{\ell^{p,\infty}(L^2)}
\norm{\wt b_N}_{\ell^{2p}(L^2)}.
$$
This leads to \eqref{b7}.
\end{proof}

As already mentioned,   Theorem~\ref{thm.b2}  was proven in \cite{BS3}
for compactly supported $b$. Hence for an arbitrary $b\in L^{2p}_{\loc} ({\bbR})$   and all $N$ we have
$$
\Delta_{p}^\pm(\Psi_N)
=
\delta_{p}^\pm(\Psi_N)
=
\frac1{2\pi}
\bigl(
A(-\infty)_\pm^{p}+A(+\infty)_\pm^{p}
\bigr)
\int_{-N}^N \abs{b(x)}^{2p}dx.
$$
Thus, to extend  Theorem~\ref{thm.b2}  to general $b$ satisfying  estimate \eqref{b1a} 
we only have to  pass here to the limit $N\to\infty$.   
Lemma~\ref{appr} 
together with the estimates \eqref{C5}, \eqref{C6}
allows us to do this. This yields \eqref{b5}.

\section*{Acknowledgements}
The authors are grateful to V.~V.~Peller and to J.~R.~Partington for useful discussions. 
Most of the work was completed during the two visits: by A.P. to the University of Rennes 1 and
by D.Y. to King's College London. The authors are grateful to both Departments and to the 
London Mathematical Society for the financial support.

\end{document}